\documentclass[11pt,letterpaper]{article}

\usepackage[OT1]{fontenc}
\usepackage[english]{babel}
\usepackage[utf8]{inputenc}
\usepackage[sc]{mathpazo}
\usepackage{amsmath,amssymb,amsfonts,mathrsfs}
\usepackage[amsmath,thmmarks]{ntheorem}
\usepackage{graphicx}
\usepackage{pdfpages}

\usepackage{multirow}
\usepackage{varioref}
\usepackage{mathtools}
\usepackage{ifthen}
\usepackage{stmaryrd}
\usepackage[h]{esvect}
\usepackage{array}
\usepackage{listings}
\usepackage{booktabs}
\usepackage{todonotes}
\usepackage[linesnumbered,lined,ruled,vlined]{algorithm2e}
\usepackage{makecell}
\usepackage{pdflscape}
\usepackage{subcaption}

\usepackage[linkcolor=black,colorlinks=true,citecolor=black,filecolor=black]{hyperref}
\usepackage{cleveref}

\numberwithin{equation}{section}

\newtheorem{theorem}{Theorem}[section]

\newtheorem{remark}[theorem]{Remark}

\newtheorem{definition}[theorem]{Definition}
\newtheorem{lemma}[theorem]{Lemma}

\newtheorem{observation}[theorem]{Observation}

\theoremstyle{nonumberplain}
\theorembodyfont{\normalfont}
\theoremsymbol{\ensuremath{\square}}
\newtheorem{proof}{Proof}
\crefname{proof}{proof}{proofs}

\renewcommand{\epsilon}{\ensuremath\varepsilon}

\renewcommand{\phi}{\ensuremath{\varphi}}

\renewcommand{\epsilon}{\ensuremath{\varepsilon}}

\renewcommand{\theta}{\ensuremath{\vartheta}}

\newcommand{\xor}{\oplus}

\newcommand{\matousek}{Matou\v{s}ek}
\newcommand{\gartner}{G\"artner}
\newcommand{\szabo}{Sz\'abo}

\DeclareMathOperator{\sign}{sign}

\title{A Characterization of the Realizable \matousek{} Unique Sink Orientations}
\author{Simon Weber\thanks{Institute of Theoretical Computer Science, Department of Computer Science, ETH~Zurich, Switzerland} \and Bernd G\"artner\footnotemark[1]}
\date{\today}

\begin{document}
\maketitle

\begin{abstract}
The \matousek{} LP-type problems were used by \matousek{} to show that the Sharir-Welzl algorithm may require at least subexponential time~\cite{matousek1994lowerbound}.
Later, \gartner{} translated this result into the language of Unique Sink Orientations (USOs) and introduced the \emph{\matousek{}~USOs}, the USOs equivalent to \matousek{}'s LP-type problems. He further showed that the Random Facet algorithm only requires quadratic time on the \emph{realizable} subset of the \matousek{}~USOs, but without characterizing this subset~\cite{gaertner2002simplex}.
In this paper, we deliver this missing characterization 
and also provide concrete realizations for all realizable \matousek{} USOs. Furthermore, we show that the realizable \matousek{} USOs are exactly the orientations arising from simple extensions of cyclic-P-matroids~\cite{fukuda2013subclass}.
\end{abstract}

\newpage
\section{Introduction}

\emph{Unique Sink Orientations (USOs)} are orientations of the $n$-dimensional hypercube graph, such that every subcube has exactly one sink. USOs were introduced to combinatorially abstract various algebraic and geometric problems, in particular the \emph{P-matrix linear complementarity problem (P-LCP)} \cite{stickney1978digraph,szabo2001usos}. If there is an instance of the P-LCP inducing a certain USO, we call the USO \emph{realizable}. We also call such an orientation a \emph{P-cube}.

The algorithmic problem associated with a USO is that of finding the unique sink of the whole cube. Its complexity remains unknown both for general USOs and for the subclass of realizable USOs. A promising candidate algorithm is \emph{Random Facet}, but also its runtime is unknown. On acyclic USOs, Random Facet requires $\exp(\Theta(\sqrt{n}))$ vertex evaluations~\cite{gaertner2002simplex}. For cyclic USOs, the upper bound is lost and only the lower bound of $\exp(\Omega(\sqrt{n}))$ remains. This lower bound is due to \matousek{} and \gartner{} and uses a concrete class of acyclic USOs that we now call the \emph{\matousek{}~USOs}~\cite{gaertner2002simplex,matousek1994lowerbound}.

\gartner{} showed that Random Facet is fast on all realizable \matousek{} USOs~\cite{gaertner2002simplex}, which means that for P-cubes, there is also no good lower bound on the runtime of Random Facet. In his proof of this statement, \gartner{} uses only a weak necessary condition for realizability. In this paper, we will characterize the subset of realizable \matousek{} USOs exactly, and show that this condition turns out to be sufficient for \matousek{}~USOs. Furthermore, we give concrete realizations for all USOs in this subset.

\emph{Oriented matroids} are a common combinatorial object to abstract properties of a wide collection of concepts, such as arrangements of {(pseudo\nobreakdash-)hyperplanes}, arrengements of vectors, or directed graphs. \emph{P-matroids} are a subclass of oriented matroids which can be viewed as a combinatorial abstraction of P-matrices, and their \emph{extensions} generalize instances of the P-LCP~\cite{todd1984matroids}. Just like a USO, an extension of a P-matroid does not have to be realizable by a P-LCP instance. There is a natural translation from P-matroid extensions to USOs~\cite{klaus2012phd}, with realizability carrying over.

Even though extensions of P-matroids and USOs are two different combinatorial abstractions of the same algebraic problem (P-LCP) with a natural link, they are mostly researched in isolation. In this work we look at the \emph{cyclic-P-matroids}, a class of realizable P-matroids introduced by Fukuda, Klaus, and Miyata~\cite{fukuda2013subclass,klaus2012phd}, and we will explore the USOs corresponding to \emph{simple extensions} of them.
In doing so, we aim to bridge the two fields of oriented matroid theory and Unique Sink Orientations.

\section{Preliminaries}
\subsection{Unique Sink Orientations}

\paragraph{Hypercube orientations.} An \emph{orientation} of a hypercube with dimension set $[n]$ is a function assigning each vertex of the cube its \emph{outmap}, i.e., the set of dimensions for which the edges adjacent to the vertex are pointing away from it. Vertices of the cube are characterized by subsets of $[n]$. More formally, with $P([n])$ being the power set of $[n]$, an orientation $o$ is a function $$o:P([n])\rightarrow P([n]),$$ such that for all vertices $v\in P([n])$ and dimensions $d\in [n]$, $d\in o(v)$ if and only if $d\not\in o(v\xor\{d\})$.

We call a subcube of a hypercube a \emph{face}. A subcube of dimension $n-1$ is also called a \emph{facet}. The set $\{v\in P([n]):d\in v\}$ is called the \emph{upper $d$-facet}, and the opposite facet is called the \emph{lower $d$-facet}.

\paragraph{Unique sink orientations.} An orientation $o$ is a USO if for each face $F$ spanned by some set of dimensions $D_F\subseteq [n]$, there is exactly one $v\in F$ with $o(v)\cap D_F = \emptyset$. Alternatively, the \szabo{}-Welzl condition~\cite{szabo2001usos} says that an orientation is a USO if and only if $$\forall v,w\in P([n]): \big(v\xor w\big)\cap \big(o(v)\xor o(w)\big)\not=\emptyset.$$
In other words, $o$ is a bijection when constrained to any face.

\paragraph{\matousek{} USOs.} A \matousek{} USO $m$ is a USO characterized by a directed graph $G=([n],E)$ with a loop $(d,d)\in E$ for every $d\in[n]$, but no other cycles. The \matousek{} USO is built from the graph using the conditions $m(\emptyset)=\emptyset$ and
\begin{gather*}
\forall d,d'\in [n],\forall v\in P([n]):\\
\Big(\big(d'\in m(v)\big) \xor \big(d'\in m(v\xor d)\big)\Big)\Leftrightarrow (d,d')\in E.
\end{gather*}
This uniquely defines $m$, as the direction of every edge in the hypercube can be determined by its relative location to the parallel edge adjacent to the vertex $\emptyset$. We call $G$ the \emph{dimension influence graph} of $m$ and if $(d,d')\in E$, we also say $d$ \emph{influences} $d'$. Intuitively, $d$ influencing $d'$ means that when walking along an edge in dimension $d$, the outmap of the vertices changes in dimension $d'$. An example of a dimension influence graph and a \matousek{}~USO can be seen in \Cref{fig:matousekexample}.

\begin{figure}[htbp]
\centering
\includegraphics[]{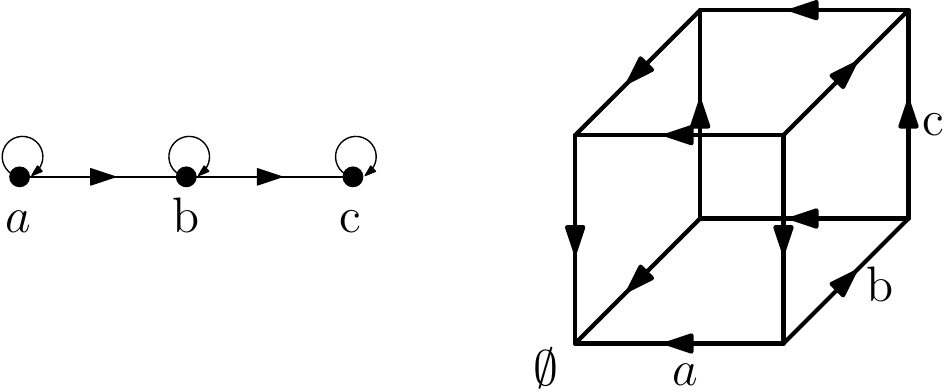}
\caption{A dimension influence graph and the \matousek{} USO characterized by it.}\label{fig:matousekexample}
\end{figure}

It can be seen rather easily that \matousek{}~USOs are indeed USOs. Assume two vertices $v,w$ fail the \szabo{}-Welzl condition, i.e., they have the same outmaps within the dimensions $v\xor w$. In this case, the induced subgraph of $G$ on the dimensions $v\xor w$ would need to have an even in-degree at every vertex, which would require $G$ to contain a cycle apart from the loops.

\paragraph{Orientation isomorphism.} Two orientations $o,o'$ are called \emph{isomorphic}, if there is a set $F\subseteq [n]$ and a permutation $\pi:[n]\rightarrow [n]$, such that for all $v\in P([n])$,
$$\pi(o(v))=o'(\pi(v\xor F)).$$
In other words, $o'$ can be obtained by mirroring $o$ along all dimensions in $F$ and then reordering the dimensions according to $\pi$. Such isomorphisms preserve realizability and the runtime of most sink-finding algorithms, including Random Facet.

For \matousek{} USOs $m$, switching the direction of all edges in some dimension $d$ is equivalent to applying the isomorphism with $\pi=\text{id}$ and $F$ such that $m(F)=\{d\}$. After either of these operations, the outmap of vertex $\emptyset$ is $\{d\}$ and the way the dimensions influence each other is unchanged, therefore they lead to the same USO. Under isomorphism, we can thus drop the condition of $m(\emptyset)=\emptyset$ for \matousek{} USOs.
We call a USO isomorphic to a \matousek{} USO a \emph{\matousek{}-type USO}.

\subsection{P-Matrix Linear Complementarity Problem}

An instance of the P-matrix linear complementarity problem (P-LCP) is given by a P-matrix $M\in\mathbb{R}^{n\times n}$ and a vector $q\in \mathbb{R}^n$. The task is to find vectors $w,z\in\mathbb{R}^n$ with $w,z\geq \mathbf{0}$, such that $w-Mz=q$, and $w$ and $z$ are complementary, i.e., $w^Tz=0$.

We call a pair of vectors $(w,z)$ with $$w-Mz=q\text{ and }\forall i: w_i=0\vee z_i=0$$ a \emph{candidate} solution. $M$ being a P-matrix guarantees that there is a unique candidate solution for each choice of the set $B=\{i:w_i=0\}$, assuming \emph{non-degeneracy} of $q$ for $M$.

A non-degenerate P-LCP instance $(M,q)$ realizes a USO $o$ in the following way: For each vertex $B\in P([n])$, $o(B)$ contains the indices $i$ for which $w_i,z_i\leq 0$ in the unique candidate solution corresponding to $B$.

\subsection{P-Matroids}

An \emph{oriented matroid} can abstractly describe the structure of many different objects, such as vector arrangements, hyperplane arrangements, or directed graphs. Furthermore, an oriented matroid can itself be described using a number of equivalent axiomatizations, such as vector axioms, circuit axioms, or chirotope axioms. In this work we focus on vector arrangements and circuit axioms. For a more complete overview over oriented matroids, we refer the reader to \cite{bjoerner1999orientedmatroids}.

\paragraph{Oriented matroids.} For a \emph{ground set} $E$, a \emph{signed set} on $E$ is a pair $(S^+,S^-)$ with $S^+\cap S^-=\emptyset$ and $S^+,S^-\subseteq E$. We call the union $\underline{S}=S^+\cup S^-$ the \emph{support} of $S$. We write $-S$ for the signed set $-S:=(S^-,S^+)$.

An oriented matroid in circuit representation is a pair $\mathcal{M}=(E,\mathcal{C})$ where $\mathcal{C}$ is a collection of signed sets on $E$, the so-called \emph{circuits} of $\mathcal{M}$.\\
$\mathcal{C}$ must satisfy the circuit axioms:\begin{itemize}
\item[(C0)] $\emptyset\not\in \mathcal{C}$
\item[(C1)] Symmetry: $\forall C\in\mathcal{C}: -C\in \mathcal{C}$
\item[(C2)] Incomparability: $\forall X,Y\in\mathcal{C}: \underline{X}\subseteq\underline{Y}\Rightarrow (X=Y \vee X=-Y)$
\item[(C3)] Weak elimination: $\forall X,Y\in\mathcal{C}$ with $X\not=-Y$, and $\forall e\in X^+\cap Y^-$, there exists a $Z\in\mathcal{C}$, such that
\begin{flalign*}
& Z^+\subseteq(X^+\cup Y^+)\backslash\{e\}\text{ and}&&\\
& Z^-\subseteq(X^-\cup Y^-)\backslash\{e\}.&&
\end{flalign*}
\end{itemize}

A \emph{basis} of $\mathcal{M}$ is an inclusion-maximal set that does not contain any circuit. All bases of an oriented matroid have the same size, called the matroid's \emph{rank}. For any basis $B$ and any element $e\in E-B$, there is a unique \emph{fundamental circuit} $C(B,q)$ with $\underline{C(B,q)}=B\cup\{q\}$ and $q\in C(B,q)^+$.

A vector configuration (or matrix) $V\in \mathbb{R}^{d\times n}$ gives rise to an oriented matroid $\mathcal{M}$ with $E=[n]$ and $\mathcal{C}$ being the collection of inclusion-minimal linear dependencies of columns of $V$, i.e., $$\mathcal{C}=\{\sign x: Vx=0\;\wedge \not\exists x':\left(\underline{x'}\subset\underline{x}\wedge Vx'=0\right)\},$$
where $\sign x$ is the signed set $\big(\{i:x_i>0\},\{i:x_i<0\}\big )$.
We say that $V$ \emph{realizes} $\mathcal{M}$.

\paragraph{P-matroids.} A \emph{P-matroid} is an oriented matroid $\mathcal{M}=(E_{2n},\mathcal{C})$ with ground set $E_{2n}=[2n]$. $E_{2n}$ is split in two parts, $S=[n]$ and $T=[2n]\backslash [n]$, $S$ being a basis of $\mathcal{M}$. Elements $i$ and $i+n$ are called \emph{complementary}, together they form a \emph{complementary pair}. For any $i$, its \emph{complementary element} is denoted by $\overline{i}$. Any set of $n$ elements not containing a complementary pair is a basis. $\mathcal{M}$ is a P-matroid if there is no \emph{almost-complementary} \emph{sign-reversing circuit}, i.e., no circuit $C\in\mathcal{C}$ such that $\underline{C}$ contains exactly one complementary pair $(i,i+n)$ and $i$ and $i+n$ have different signs in $C$~\cite{todd1984matroids}.

An \emph{extension} $\hat{\mathcal{M}}=(\hat{E},\hat{\mathcal{C}})$ of a matroid $\mathcal{M}=(E,\mathcal{C})$ is an oriented matroid with ground set $\hat{E}=E\cup \{q\}$, such that the \emph{deletion minor} $\hat{\mathcal{M}}\backslash q=(E,\{X:X\in \hat{\mathcal{C}}\text{ and }X_q=0\})$ is equal to $\mathcal{M}$.

Given an extension of a P-matroid $\hat{M}=(\hat{E_{2n}},\hat{\mathcal{C}})$, one can obtain the associated USO $u$ using the following procedure~\cite{klaus2012phd}: For each vertex $v\in P([n])$, determine the fundamental circuit $$C_v := C(\{i:i\in [n]-v\}\cup\{i+n:i\in v\},\;q)$$
and if $i\in C_v^-$ or $i+n\in C_v^-$, add $i$ to $u(B)$.

For a P-matrix $M\in\mathbb{R}^{n\times n}$, the \emph{associated P-matroid} is the oriented matroid realized by $\begin{bmatrix}I_n & -M\end{bmatrix}$. For a P-LCP instance $(M,q)$, the \emph{associated extension of a P-matroid} is realized by $\begin{bmatrix}I_n & -M & -q\end{bmatrix}$.

If one is given a $(2n+1)\times n$-matrix $V$ realizing the extension of a P-matroid, a P-LCP instance with the same associated P-matroid extension can be found using \begin{equation}\label{eqn:translation}(M,q):=(-V_S^{-1}V_T,-V_S^{-1}V_{2n+1}).
\end{equation}

\section{Cyclic-P-Matroids}

The \emph{alternating matroid} $\mathcal{A}^{n,r}$ is an oriented matroid of rank $r$ and on ground set $E_n=[n]$. It is realized by $n$ sequential points on the moment curve, i.e., by the matrix
$$V\in\mathbb{R}^{r\times n}:=\begin{pmatrix}
1 & 1 & \hdots & 1 \\
x_1 & x_2 & \hdots & x_n \\
x_1^2 & x_2^2 & \hdots & x_n^2 \\
\vdots & \vdots & & \vdots \\
x_1^{r-1} & x_2^{r-1} & \hdots & x_n^{r-1}
\end{pmatrix}$$
for $x_1 < x_2 < \cdots < x_n$.

The alternating matroids have been studied deeply and their structure is well-understood. For our purposes here, it is only important to know that they are uniform and that the signs of the non-zero elements in each circuit alternate along the natural order from $1$ to $n$~\cite{cordovil2000cyclicmatroid}.

A cyclic-P-matroid is a P-matroid that is reorientation equivalent to the alternating matroid $\mathcal{A}^{2n,n}$~\cite{klaus2012phd}. Formally, there exists a set $F\subseteq E_{2n}$ and a permutation $\pi$ of $E_{2n}$ such that $\mathcal{M}={}_{-F}(\pi^{-1}\cdot \mathcal{A}^{2n,n})$. This means that $\mathcal{M}$ can be obtained by relabeling the elements of $E_{2n}$ according to $\pi^{-1}$ in all circuits of $\mathcal{A}^{2n,n}$, and then flipping the sign of all elements in $F$ in all resulting circuits.

As $\mathcal{A}^{2n,n}$ (and therefore also $\mathcal{M}$) is uniform, any set $S\subset E_{2n}$ of size $n$ is a base of $\mathcal{M}$, and thus any set $S'\subset E_{2n}$ of size $n+1$ is the support of some fundamental circuit. The signs of the elements can be read off by ordering $S'$ according to $\pi$, giving them alternating signs, and finally flipping the sign of all elements that are in $F$. This process can be seen in \Cref{fig:readingoff}. Note that there are always two circuits for each $S'$, each being the negation of the other.

\begin{figure}[htbp]
\centering
\includegraphics[scale=0.6]{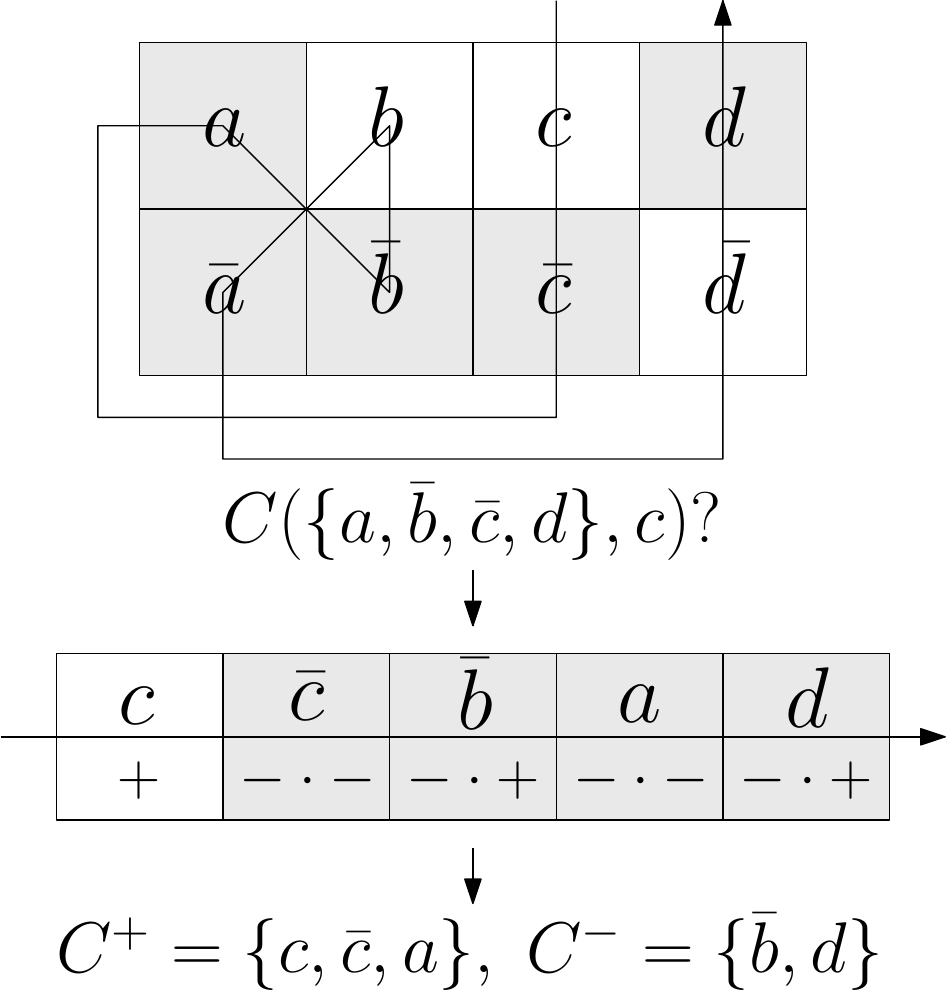}
\caption{Reading off the signs of a fundamental circuit in a cyclic-P-matroid. The arrow indicates the order of the elements in $\pi$. The shaded elements are those in $F$.}\label{fig:readingoff}
\end{figure}

Fukuda, Klaus, and Miyata \cite{fukuda2013subclass,klaus2012phd} characterized for which choices of $\pi$ and $F$ this leads to $\mathcal{M}$ being a P-matroid. Sadly, their characterization also allows for some choices for which $\mathcal{M}$ is not a P-matroid. We therefore provide a corrected version of \cite[Theorem 4.2] {fukuda2013subclass}:

\begin{theorem}\label{thm:correcteddef}
An oriented matroid $\mathcal{M}$ on $E_{2n}$, where $\pi\cdot({}_{-F}\mathcal{M})=\mathcal{A}^{2n,n}$ for some permutation $\pi$ of $E_{2n}$ and $F\subseteq E_{2n}$, is a P-matroid if and only if:
\begin{itemize}
\item $\pi$ is such that for all $e\in E_{2n}$ with $\pi(e)<\pi(\bar{e})$, and all $f\in E_{2n}$,
$$\pi(f)\in [\pi(e),\pi(\bar{e})]\Longleftrightarrow \pi(\bar{f})\in [\pi(e),\pi(\bar{e})].$$
\item and $F$ is such that for every $e\in E_{2n}$, exactly one of $e$ and $\bar{e}$ is in $F$ if
$$\frac{|\pi(e)-\pi(\bar{e})|-1}{2}$$
is even, and both or none in $F$ otherwise.
\end{itemize}
\end{theorem}
In Theorem 4.2 of \cite{fukuda2013subclass}, the condition on $\pi$ was such that $|\pi(e)-\pi(\bar{e})|$ was required to be odd, but the intervals $[\pi(e),\pi(\bar{e})]$ were not required to contain either both or none of $\pi(f)$ and $\pi(\bar{f})$.	

\begin{proof}
We will examine the almost-complementary circuits to prove the two directions individually. Recall that an oriented matroid is a P-matroid if and only if it contains no sign-reversing almost-complementary circuit.

\begin{enumerate}
\item[if] Assume $\pi$ and $F$ fulfill the conditions in \Cref{thm:correcteddef}. Let $\underline{C}$ be the support of some almost-complementary circuit $C$, and let $e$ be the element for which both $e,\bar{e}\in \underline{C}$. As $C$ is almost-complementary, the number of elements of $\underline{C}$ between $e$ and $\bar{e}$ according to $\pi$ must be $\frac{|\pi(e)-\pi(\bar{e})|-1}{2}$, as for each $f\in E_{2n}$, exactly one of $f$ and $\bar{f}$ is contained in $\underline{C}$. The signs of $e$ and $\bar{e}$ in $C$ must therefore be the same, as $F$ flips one of their signs exactly if the parity of their positions in $\underline{C}$ according to $\pi$ is different.\\
As this holds for all almost-complementary circuits, $\mathcal{C}$ contains no almost-complementary sign-reversing circuit, and $\mathcal{M}$ must be a P-matroid.
\item[only if] Assume $\pi$ and $F$ do not fulfill the conditions in \Cref{thm:correcteddef}. Then either $\pi$ does not fulfill the first condition, or $F$ does not fulfill the second.\\
In the first case, there must be elements $e,f$ such that the interval $[\pi(e),\pi(\bar{e})]$ contains exactly one of $\pi(f)$ and $\pi(\bar{f})$. Let $D$ be a maximal complementary set $D\subset E_{2n}-\{e,\bar{e},f,\bar{f}\}$ and compare the fundamental circuits $C_1:=C(D\cup\{\bar{e},f\},e)$ and $C_2:=C(D\cup\{\bar{e},\bar{f}\},e)$. The sign of $e$ must be the same in $C_1$ and $C_2$, while the sign of $\bar{e}$ must be different in both. We conclude that at least one of $C_1$ and $C_2$ must be an almost-complementary sign-reversing circuit.\\
In the second case, $F$ must violate the condition for some $e$. For any maximal complementary set $D\subset E_{2n}-\{e\}$, the fundamental circuit $C(D,\bar{e})$ must be an almost-complementary sign-reversing circuit, by the same arguments as in the proof of the ``if'' direction.\\
We conclude that $\mathcal{C}$ must contain an almost-complementary sign-reversing circuit, therefore $\mathcal{M}$ is not a P-matroid.
\end{enumerate}
\end{proof}

With the conditions laid out here, the permuted elements form a balanced parentheses expression, with each pair $e,\bar{e}$ acting as two matching parentheses. In this view, we can create a directed graph $G_\pi=([n],E)$, where $(i,j)\in E$ if the elements $j,\bar{j}$ are contained within the parenthesis formed by $i,\bar{i}$ (also $(i,i)\in E$ for all $i$).

\begin{lemma}\label{lemma:arborescence}
For $\pi$ fulfilling the conditions of \Cref{thm:correcteddef}, $G_\pi$ is the transitive closure of a \emph{branching}
\footnote{An \emph{arborescence} is a directed analogue of a rooted tree with all edges pointing away from the root. A \emph{branching}, or a forest of arborescences, is the union of finitely many disjoint arborescences.}.
Furthermore, the transitive closure of any branching is the graph $G_\pi$ for some valid $\pi$.
\end{lemma}
\begin{figure}[htbp]
\centering
\includegraphics[]{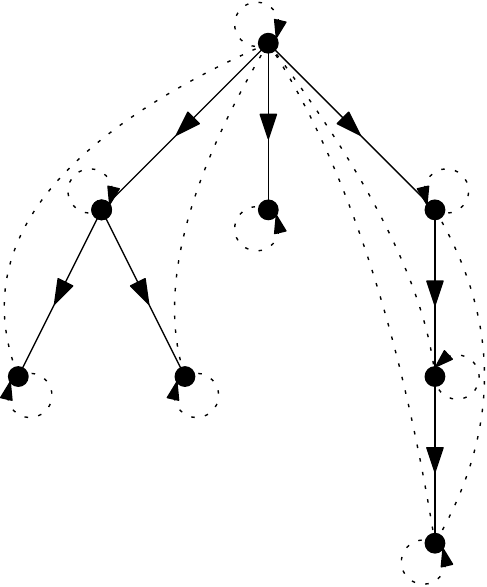}
\caption{An arborescence with the \emph{transitive edges} added by taking the transitive closure indicated with dotted lines. The non-transitive edges are called \emph{direct edges}.}\label{fig:arborescence}
\end{figure}
\begin{proof}
We prove both directions independently.

First, we show that for any valid $\pi$, $G_\pi$ is the transitive closure of a branching. Clearly, $G_\pi$ must be transitive, as it is built from the transitive relation of relative containment of intervals. We can assign each interval (and therefore each vertex in $G_\pi$) a depth indicating the number of intervals it is contained in. Whenever an interval is contained in two other intervals, one of these two intervals must contain the other. Therefore an interval of depth $k$ is contained in exactly one interval of depth $l$ for any $l<k$ and thus each vertex in $G_\pi$ has exactly one incoming edge from any lower depth. We conclude that $G_\pi$ must be the transitive closure of a branching.

Now we show that any transitive closure of a branching can be realized by some valid $\pi$. Each arborescence is an independent part of the permutation, ordered arbitrarily. The interval corresponding to the root of the arborescence encompasses the intervals for all its descendants. The descendants of the root again form a transitive closure of a branching, and can be converted into a permutation recursively.
\end{proof}

\section{Simple Extensions Induce \matousek{} USOs}

A \emph{simple extension} of a cyclic-P-matroid is an extension that can be realized by $q$ being an additional point $(1,x_q,...,x_q^{n-1})$ on the moment curve. 
\begin{definition}[\protect{\cite[Proposition 8.18]{klaus2012phd}}]
A simple extension $\hat{\mathcal{M}}$ of a cyclic-P-matroid is an oriented matroid $\hat{\mathcal{M}}=(\hat{E_{2n}},\hat{\mathcal{C}})$ with $\pi\cdot({}_{-F}\hat{\mathcal{M}})=\mathcal{A}^{2n+1,n}$ for some permutation $\pi$ of $\hat{E_{2n}}=E_{2n}\cup \{q\}$ with $\pi(q)=q$ and $F\subseteq\hat{E_{2n}}$, such that $\pi$ and $F$ restricted to $E_{2n}$ fulfill the conditions of \Cref{thm:correcteddef}.
\end{definition}

We will first look at the USOs that can be realized by simple extensions of cyclic-P-matroids with $q$ being the last element in the permutation, i.e., $x_q > x_{2n}$.

\subsection{$q$ at the End of $\pi$}

If $q$ is at the end of the permutation, there is always the same number ($n$) of elements before $q$ in any circuit $C$. When reading off the fundamental circuit $C(B,q)$, we will therefore always assign the first element the same sign (before possibly flipping it due to it being in $F$).

\begin{lemma}\label{lemma:qatend}
Let $\hat{M}$ be a simple extension of a cyclic-P-matroid with $\pi\cdot({}_{-F}\hat{\mathcal{M}})=\mathcal{A}^{2n+1,n}$, with $q$ last in $\pi$. The USO $u$ associated with $\hat{M}$ is a \matousek{}-type USO with dimension influence graph $G_{\pi'}$, where $\pi'$ is $\pi$ constrained to $E_{2n}$.
\end{lemma}
\begin{proof}
We will prove that for any vertex, going along the edge in some dimension $i$, the outmap changes exactly in the dimensions $j$ for which $(i,j)\in E(G_{\pi'})$. As $G_{\pi'}$ is acyclic (apart from the loops), this proves that the USO is a \matousek{}-type USO with dimension influence graph $G_{\pi'}$.

Let $B\in P([n])$ be some vertex of the USO, and $C(B,q)$ its corresponding fundamental circuit. Let $i\in [n]$ be some dimension. Note that $\underline{C(B,q)}$ contains either $i$ or $i+n$. Assume w.l.o.g. that $i\in \underline{C(B,q)}$, and $\pi'(i)<\pi'(i+n)$.

We can now split the set $(B\cup q)\backslash i$ in three parts along $\pi'$: The part before $i$, the part between $i$ and $i+n$, and the part after $i+n$. We will analyze what happens to the sign of the elements in each part when $i$ is removed from the circuit and replaced by $i+n$.

The sign of an element $j$ of the first part will not change, as neither $i$ nor $i+n$ is located before $j$. This means that the number of elements of the circuit before $j$ and thus its sign in the circuit stays the same. The same holds for all elements of the third part, as both $i$ and $i+n$ are located before them.

For an element $j$ of the middle part, the number of elements in the circuit coming before $j$ decreases by one when $i$ is exchanged with $i+n$, and therefore its sign in the circuit flips.

This shows that $u(B\xor i)$ differs from $u(B)$ in exactly the dimensions coming between $i$ and $i+n$ in $\pi'$ (and $i$ itself), which is exactly the set of out-neighbours of $i$ in $G_{\pi'}$.

We conclude that $u$ is a \matousek{}-type USO with dimension influence graph $G_{\pi'}$.
\end{proof}

Putting \Cref{lemma:arborescence,lemma:qatend} together, we know that any \matousek{}-type USO with dimension influence graph being the transitive closure of a branching is realizable. We will now show that these are in fact all realizable \matousek{}-type USOs.

\begin{lemma}\label{lemma:forbidden}
If the dimension influence graph of a \matousek{} USO contains any of the two graphs $G_1$ or $G_2$ in \Cref{fig:forbidden} as an induced subgraph, the USO is not realizable.
\end{lemma}
\begin{proof}
Any $3$-dimensional face spanned by the dimensions whose vertices induce one of these subgraphs is isomorphic to either $u_1$ or $u_2$ (see \Cref{fig:forbidden}). Both $u_1$ and $u_2$ do not contain three vertex-disjoint paths from their source to their sink, and are therefore not realizable due to failing the Holt-Klee condition~\cite{gaertner2008grids,holt1999klee}. As all faces of a realizable USO are realizable, the whole \matousek{} USO cannot be realizable either.
\end{proof}

\begin{figure}[htbp]
\centering
\includegraphics[]{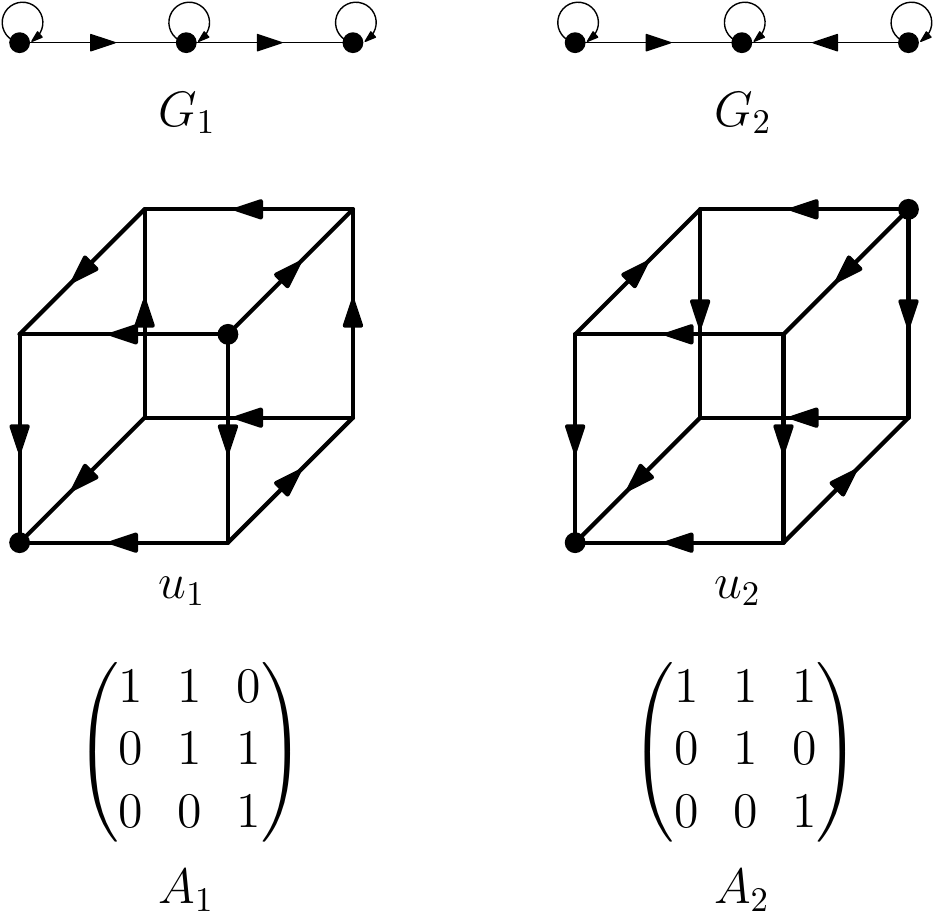}
\caption{The two forbidden subgraphs $G_1,G_2$, as well as their associated USOs and \matousek{} matrices as used in \cite{gaertner2002simplex}.}\label{fig:forbidden}
\end{figure}

\begin{lemma}
Every \matousek{} USO with a dimension influence graph which is not the transitive closure of a branching is not realizable.
\end{lemma}
\begin{proof}
Note that the dimension influence graph $G$ is always acyclic. Assume $G$ is not the transitive closure of a branching.

Assume $G$ is not transitive. In this case $G$ contains three vertices $x$, $y$, and $z$, with $(x,y),(y,z)\in E$, but $(x,z)\not\in E$. This means the graph contains the forbidden subgraph $G_1$.

We thus assume $G$ is transitive. As $G$ is not the transitive closure of a branching, there must be a vertex $x$ that has an incoming edge from at least two vertices $p_1$ and $p_2$, where neither $p_1$ has an edge to $p_2$ nor vice-versa. In this case, $x$, $p_1$, and $p_2$ together form the forbidden subgraph $G_2$.

In either case, by \Cref{lemma:forbidden}, the \matousek{} USO with dimension influence graph $G$ is not realizable.
\end{proof}

\begin{theorem}
The realizable \matousek{} USOs are exactly those whose dimension influence graph does not contain either of the forbidden graphs $G_1$ and $G_2$ as induced subgraphs. Furthermore, they can be realized by a simple extension of a cyclic-P-matroid with $q$ as the last element of $\pi$.
\end{theorem}

As a corollary, we get a simple way to realize each realizable \matousek{} USO by first determining the simple extension of a cyclic-P-matroid realizing it, and then applying \Cref{eqn:translation} to the matrix of $2n+1$ (permuted and negated) points on the moment curve realizing that matroid.

\subsection{$q$ in the Middle of $\pi$}

Having understood the USOs stemming from simple extensions with $q$ at the end of $\pi$, we would like to understand the influence of moving $q$ towards the front of $\pi$. As we will see, this operation does not change the set of possible USOs.

\begin{observation}\label{observation:pushingisflipping}
Let $\hat{\mathcal{M}}$ be some simple extension of a cyclic-P-matroid with corresponding USO $u$. If $q$ is pushed towards the start of $\pi$ past some element $i\in [n]$, the direction of all edges between vertices in the lower $i$-facet is flipped.
\end{observation}

Similarly, if $q$ is pushed past an element $i+n\in [n]$, the direction of all edges between vertices in the upper $i$-facet is flipped.

\begin{observation}\label{observation:flippingchangesgraph}
If all edges of some $i$-facet of a \matousek{}-type USO $u$ are flipped, this changes whether $(i,j)$ is an edge in the dimension influence graph of $u$ for all $j\in [n]\backslash i$.
\end{observation}

When moving $q$ towards the front of $\pi$, facets are flipped from the back of $\pi'$ towards the front. The flipped facets therefore form a suffix of $\pi'$. Note that if both the upper and lower facet of some dimension are flipped, this operation can be ignored. The sets of dimensions for which exactly one facet can be flipped at the same time therefore form a path in $G_{\pi'}$ starting at some root, following only direct edges and not using transitive edges (see \Cref{fig:arborescence}).

\begin{lemma}\label{lemma:flipping}
Let $G$ be the transitive closure of some branching. Let $S$ be the set of vertices on some directed path from some root of $G$, such that there is no $v\not\in S$ with vertices $u,w\in S$ and $(u,v),(v,w)\in E$.

Flipping whether $(s,t)\stackrel{?}{\in} E$ for all $s\in S$ and all $t\not=s$ results in some graph $G_S$ which is also the transitive closure of a branching.
\end{lemma}

We leave this lemma without formal proof, but intuitively, this operation removes the vertices of $S$ from their descendants, reverses their order, and adds them as a parent to all other vertices of the same depth, as can be seen in \Cref{fig:pathreversal}.

\begin{figure}[htbp]
\centering
\includegraphics[]{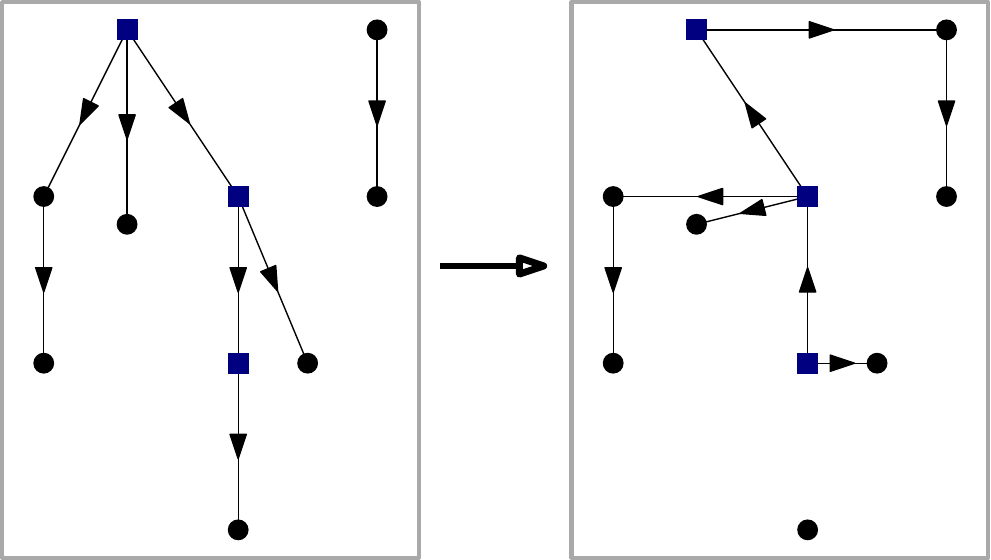}
\caption{Flipping whether $(s,t)\in E$ for all $s$ in some path (blue square vertices) starting at a root and all $t\not=s$. The transitive edges are left out for clarity.}
\label{fig:pathreversal}
\end{figure}

Putting together \Cref{observation:pushingisflipping}, \Cref{observation:flippingchangesgraph}, and \Cref{lemma:flipping}, we get our following main result.

\begin{theorem}
The set of USOs that can stem from simple extensions of cyclic-P-matroids is the set of realizable \matousek{}-type USOs.
\end{theorem}

We end with the following remark about flipping facets in non-realizable \matousek{} USOs:
\begin{remark}
The operation of flipping facets can be applied to non-realizable \matousek{} USOs too. If this makes the dimension influence graph cyclic, the result is \textbf{not} a USO. If the dimension influence graph stays acyclic, the result is simply a non-realizable \matousek{}-type USO again.
\end{remark}

\newpage
\section{Conclusion}

We have corrected the characterization of cyclic-P-matroids from \cite{fukuda2013subclass,klaus2012phd} and have then shown that the simple extensions of cyclic-P-matroids can induce all realizable \matousek{} USOs and USOs isomorphic to them. By doing this, we have also characterized the set of realizable \matousek{} USOs, and found concrete realizations for all of them.

In the future, it would be interesting to also understand non-simple extensions of cyclic-P-matroids and their associated USOs. These extensions could be computed from the set of uniform extensions of the alternating matroid $\mathcal{A}^{2n,n}$. As Ziegler showed \cite{ziegler1993bruhat}, this set is isomorphic to the higher Bruhat order $B(2n,n)$ which we cannot efficiently enumerate yet. Maybe we can find properties of non-simple extensions of cyclic-P-matroids without having to solve this problem.

Furthermore, we would like to know whether the Holt-Klee condition is also necessary (or even necessary and sufficient) for USOs stemming from P-matroid extensions.

\bibliography{literature.bib}
\bibliographystyle{plain}

\end{document}